\documentclass{amsart}
\usepackage{amssymb}
\usepackage{amsmath}
\usepackage{amsthm}
\usepackage{mathrsfs}
\usepackage{hyperref}

\title[Bilipschitz Equivalence of Trees and Hyperbolic Fillings]{Bilipschitz Equivalence of Trees and Hyperbolic Fillings} 
\author{Jeff Lindquist}
\email{JLindquist@math.ucla.edu}
\address{Department of Mathematics, University of California, Los Angeles,  Box 95155, Los Angeles, CA, 90095-1555, USA}
\thanks{The author was partially supported by NSF grants DMS-1506099 and DMS-1162471.}

\newtheorem{thm}{Theorem}[section]
\newtheorem{lemma}[thm]{Lemma}
\newtheorem{cor}[thm]{Corollary}
\theoremstyle{definition}
\newtheorem{defn}[thm]{Definition}
\newtheorem*{ack}{Acknowledgements}
\newtheorem{remark}[thm]{Remark}

\newcommand{\N}{\mathbb{N}}
\newcommand{\Z}{\mathbb{Z}}

\newcommand{\R}{\mathbb{R}}

\DeclareMathOperator{\diam}{diam}

\makeatletter
\def\blfootnote{\gdef\@thefnmark{}\@footnotetext}
\makeatother

\begin{document}
\maketitle


\begin{abstract}
We combine conditions found in \cite{Wh} with results from \cite{MPR} to show that quasi-isometries between uniformly discrete bounded geometry spaces that satisfy linear isoperimetric inequalities are within bounded distance to bilipschitz equivalences.  We apply this result to regularly branching trees and hyperbolic fillings of compact, Ahlfors regular metric spaces.
\end{abstract}

\blfootnote{This work is based on work from the author's thesis.  }

\section{Introduction}  

In this note we combine results from \cite{Wh} with results from \cite{MPR} both to generalize a theorem of Papasoglu \cite{Pa} and to prove that the vertex sets of hyperbolic fillings of quasisymmetric, compact, Ahlfors regular metric spaces are bilipschitz equivalent.  Papasoglu in \cite{Pa} proves that (the vertices of) $k$-ary homogeneous trees are bilipschitz equivalent whenever $k \geq 3$.  A map between metric spaces $f \colon (X,d_X) \to (Y, d_Y)$ is a {\em bilipschitz equivalence} if it is a bijection such that there exists a constant $C > 0$ such that for all $x, x' \in X$ we have 
\begin{equation*}
\frac{1}{C} d_X(x, x') \leq d_Y(f(x), f(x')) \leq C d_X(x, x').
\end{equation*}
To view a connected graph $X = (V_X, E_X)$ as a metric space, we use the {\em graph metric}.  This means each edge of $X$ is taken to be isometric to an interval of length $1$.  For $x, x' \in V_X$, it follows that the quantity $d_X(x,x')$ is the fewest number of edges required to connect $x$ to $x'$.

Bilipschitz equivalence is a strong property that is not immediate in many situations. One has the weaker notion of a quasi-isometry which is a map that is bilipschitz at large scales.  More formally, a map between metric spaces $f \colon (X,d_X) \to (Y, d_Y)$ is a {\em quasi-isometry} if there exist constants $C, D > 0$ such that for all $x, x' \in X$ we have
\begin{equation*}
\frac{1}{C} d_X(x, x') - D \leq d_Y(f(x), f(x')) \leq C d_X(x, x') + D
\end{equation*}
and such that every point in $y$ is within distance $C$ of $f(X)$.  A natural question to ask is whether a quasi-isometry can be promoted to a bilipschitz equivalence under the right conditions.  A positive answer is given  by Whyte \cite{Wh} who showed that a quasi-isometry between $UDBG$ spaces is within bounded distance to a bilipschitz equivalence if a certain homological condition holds.  Here a $UDBG$ space is a metric space that is uniformly discrete with bounded geometry.  A metric space $(X,d)$ is {\em uniformly discrete} if there is a constant $c>0$ such that for all $x, x' \in X$ with $x \neq x'$ we have $d(x, x') > c$.  A metric space $(X,d)$ is said to have {\em bounded geometry} if it is uniformly discrete and, for all $r>0$ there is a constant $N_r > 0$ such that for all $x \in X$ we have $|B(x, r)| \leq N_r$.  Here and elsewhere, if $A$ is a set then $|A|$ denotes the cardinality of $A$.

Whyte's results involve boundary estimates which are reminiscent of linear isoperimetric inequalities.  Let $X = (V_X, E_X)$ be a graph.  For a subset $A \subseteq V_X$, we define the boundary of $A$ as $\partial A = \{ x \in V_X : x \notin A \text{ and } d(x,A) \leq 1 \}$.  To say $X$ (or $V_X$) satisfies a {\em linear isoperimetric inequality} means there is a constant $C > 0$ such that for all finite $A \subseteq V_X$ we have $|A| \leq C |\partial A|$.  We apply the results in \cite{Wh} to prove the following theorem.  

\begin{thm}
\label{main Theorem}
Let $X$ and $Y$ be connected graphs with the graph metric.  Suppose $X$ and $Y$ are quasi-isometric,  have bounded degree, and satisfy linear isoperimetric inequalities.  Then, the vertex sets $V_X$ and $V_Y$ are bilipschitz equivalent.  Moreover, if $f \colon X \to Y$ is a quasi-isometry, then there is a bilipschitz equivalence within bounded distance of $f|_{V_X}$.
\end{thm}

  This theorem allows us to generalize Papasoglu's result to more exotic trees that satisfy linear isoperimetric inequalities. Such trees were studied in the work of Mart\'inez-P\'erez and Rodr\'iguez \cite{MPR}.  Their results, together with a quasisymmetric characterization from \cite{DS}, yield the following corollary.

\begin{cor}
\label{Cor 1}
Let $X$ and $Y$ be rooted, pseudo-regular, visual trees of bounded degree with the graph metric.  Then, $V_X$ and $V_Y$ are bilipschitz equivalent.
\end{cor}

Here a tree is rooted if it has a specified ``root'' vertex, pseudo-regular if it branches regularly, visual if it does not have arbitrarily long ``dead ends'', and of bounded degree if there is a uniform bound on the number of edges connecting to any particular vertex.  For precise definitions we refer the reader to Section \ref{Definitions and Preliminaries}.

A metric space $(X,d)$ is Ahlfors $Q$-regular if, for $\mu$ the Hausdorff $Q$ measure induced by $d$, there are constants $c, C > 0$ such that for all $0 < r \leq \diam(X)$ we have $c r^Q \leq \mu(B(x,r)) \leq C r^Q$.  Theorem \ref{main Theorem} has another corollary when one considers hyperbolic fillings of compact, Ahlfors regular metric spaces.

\begin{cor}\label{Cor 2}
Let $(Z, d_Z)$ and $(W, d_W)$ be quasisymmetrically equivalent, compact, Ahlfors regular metric spaces.  Let $X = (V_X, E_X)$ and $Y = (V_Y, E_Y)$ be hyperbolic fillings of $Z$ and $W$.  Then, $V_X$ and $V_Y$ are bilipschitz equivalent.
\end{cor}

A homeomorphism $f \colon (X, d_X) \to (Y, d_Y)$ between metric spaces is a {\em quasisymmetry} if there exists a homeomorphism $\eta \colon [0,\infty) \to [0,\infty)$ such that for all $x, y, z \in X$, we have
\begin{equation*}
\frac{d_Y(f(x), f(y))}{d_Y(f(x), f(z))} \leq \eta \biggl(\frac{d_X(x,y)}{d_X(x,z)} \biggr).
\end{equation*}
Two metric spaces are quasisymmetrically equivalent if there is a quasisymmetry from one onto the other.  The identity map from a space to itself is a quasisymmetry, so as a special case of Corollary \ref{Cor 2} we see that any two hyperbolic fillings of a given compact, Ahlfors regular metric space are bilipschitz equivalent.

%

Hyperbolic fillings are graph appromixations of metric spaces formed by covering the metric space with specific balls and connecting two balls with an edge if they overlap.  For precise constructions, we refer the reader to \cite{BuS}, \cite{BP}, and \cite{Li}.  It follows from the work in \cite{MPR} that hyperbolic fillings of these spaces satisfy a linear isoperimetric inequality and thus quasi-isometries may be promoted to bilipschitz equivalences in this setting.  As fillings of quasisymmetric spaces are quasi-isometric, the corollary follows.

Section \ref{Definitions and Preliminaries} contains the relevant definitions and preliminaries for the rest of the paper.  In Section \ref{sec results}, we state the relevant results from \cite{Wh} and \cite{MPR}.  In Section \ref{applications} we prove the results stated in the introduction, namely Theorem \ref{main Theorem} and its two corollaries.

\begin{ack}
The author thanks Mario Bonk for interesting discussions on the subject matter.
\end{ack}


\section{Definitions and Preliminaries} \label {Definitions and Preliminaries}
We make precise more of the definitions in the introduction.

For one of our applications we consider metric spaces which are rooted trees with some additional properties.  

\begin{defn}
A {\em rooted graph} is a graph with a distinguished vertex $v_0$.  If the graph is also a tree, we call this a {\em rooted tree}.
\end{defn}

\begin{defn} 
A rooted graph $T$ is {\em visual} if there is a constant $C > 0$ such that for every vertex $v \in V$, there is an infinite geodesic ray $I$ with endpoint $v_0$ such that $d(v, I) \leq C$.  In the language of \cite{MPR}, this means $v_0$ is a {\em pole} of $T$.
\end{defn}

\begin{defn}
The {\em degree} of a vertex $v$ is the number of edges $e$ eminating from $v$, which we write as $\deg(v)$.  A graph has {\em bounded degree} if there is a $\mu > 0$ such that for all $v \in V$ we have $\deg(v) \leq \mu$.
\end{defn}

In \cite{MPR} the combinatorial Cheeger isoperimetric constant $h(\Gamma)$ of a connected graph $\Gamma$ is defined.  This constant quantifies the existence of an isoperimetric inequality in $\Gamma$.

\begin{defn}
Given a connected graph $\Gamma = (V,E)$, we define the {\em combinatorial Cheeger isoperimetric constant} of $\Gamma$ as $h(\Gamma) = \inf_A |\partial A| / |A|$ where $A$ ranges over nonempty finite subsets of $V$.
\end{defn}

\begin{remark}
The graph $\Gamma$ satisfies a linear isoperimetric inequality if and only if $h(\Gamma) > 0$.
\end{remark}

We also wish to impose a condition on trees that forces regular branching.  In \cite{MPR} there is such a condition which they call {\em pseudo-regularity}.  To fully define this we need some notation, also borrowed from \cite{MPR}.  Given a tree $T$ and points $x,y \in T$, we let $[xy]$ denote the (unique) geodesic connecting $x$ to $y$.  For a fixed point $v \in T$ and any point $x \in T$, we define
\begin{equation*}
T_x^v = \{y \in T : x \in [vy]\}.
\end{equation*}

In the following, $S(v_0,t)$ is the sphere of radius $t$ centered at $v_0$ in the graph metric.

\begin{defn}
Given a rooted, visual tree $(T,v_0)$ and $K>0$ we say $(T,v_0)$ is {\em K-pseudo-regular} if for every $t \in \N$ and every $a \in S(v_0,t)$, there exist at least two points in $S(v_0, t+K) \cap T_a^{v_0}$.  We say $(T,v_0)$ is {\em pseudo-regular} if it is {\em K-pseudo-regular} for some $K$.
\end{defn}

\begin{remark}
Even though pseudo-regular trees are rooted and visual by definition, we will refer to rooted, pseudo-regular, visual trees for emphasis.
\end{remark}

We will use the {\em end space} definition from \cite{MPR} of the boundary at infinity of a tree.  For this, let $(T,v_0)$ be a rooted tree.

\begin{defn}
The {\em end space} of the rooted tree $(T, v_0)$ is
\begin{equation*}
end(T, v_0) = \{ F \colon [0,\infty) \to T : F(0) = v_0 \text{ and } F \text{ is an isometric embedding}\}.
\end{equation*}
We define the Gromov product at infinity of two elements $F,F' \in end(T,v_0)$ as 
\begin{equation*}
(F|F')_{v_0} = \sup \{t \geq 0 : F(t) = F'(t) \}.
\end{equation*}
We then define a metric $d = d_{v_0}$ on $end(T,v_0)$ by $d(F,F') = e^{-(F|F')_{v_0}}$.  Here $d$ is actually an ultrametric.  This means if $F,G,H \in end(T,v_0)$, then $d(F,H) \leq \max(d(F,G),d(G,H))$.  We often write $\partial_\infty T$ for $end(T,v_0)$.
\end{defn}

In \cite{DS}, there is a characterization of metric spaces that are quasisymmetrically equivalent to the standard $1/3$ Cantor set, denoted here as $C_{1/3}$.

\begin{thm}[\cite{DS}, Theorem 15.11]\label{Cantor qs}
A compact metric space $(X,d)$ is quasisymmetrically equivalent to $C_{1/3}$ if it is bounded, complete, doubling, uniformly perfect, and uniformly disconnected.
\end{thm}

From \cite[Definition 15.1]{DS}, a metric space $(X,d)$ is {\em uniformly disconnected} if there is a constant $C > 0$ such that for each $x \in X$ and $r > 0$ there is a closed subset $A \subseteq X$ with $B(x, r/C) \subseteq A \subseteq B(x,r)$ and $d(A, X \setminus A) \geq r/C$.  Ultrametric spaces are uniformly disconnected by \cite[Proposition 15.7]{DS}.

From \cite[Definition 3.17]{MPR}, a metric space $(X,d)$ is {\em uniformly perfect} if there are constants $C > 1$ and $R > 0$ such that for every $x \in X$ and $0 < r < R$ there is an $x' \in X$ with $r/C < d(x,x') \leq r$.  

A metric space $(X,d)$ is {\em doubling} if there exists a constant $N > 0$ such that every ball $B(x,r) \subseteq X$ can be covered by at most $N$ balls of radius $r / 2$.

\begin{remark} \label {compact free} Complete, bounded, doubling metric spaces are compact as bounded, doubling metric spaces are totally bounded. 
\end{remark}

For Corollary \ref{Cor 1} we are concerned with rooted, pseudo-regular, visual trees of bounded degree.  We now prove the end spaces of such trees satisfy the criteria in Theorem \ref{Cantor qs}.

\begin{lemma} \label{Tree qs}
Let $(T,v_0)$ be a rooted, pseudo-regular, visual tree with bounded degree. Then, $\partial_\infty T$ is quasisymmetrically equivalent to $C_{1/3}$.
\end{lemma}

\begin{proof}
With Remark \ref{compact free} and \cite[Theorem 15.11]{DS} it suffices to show $\partial_\infty T$ is bounded, complete, doubling, uniformly perfect, and uniformly disconnected. 

From \cite[Proposition 3.3]{MPR} $\partial_\infty T$ is a complete, bounded ultrametric space.  Thus, $\partial_\infty T$ is also uniformly disconnected by \cite[Proposition 15.7]{DS}. 

The fact that $\partial_\infty T$ is uniformly perfect follows from the fact that $(T,v_0)$ is pseudo-regular.  This is proven in \cite[Proposition 3.20]{MPR}.  

We show $\partial_\infty T$ is doubling.  Let $B(F,r)$ be a ball in $\partial_\infty T$.  Let 
\begin{equation*}
M = \sup \{m : F(k)=G(k)\text{ for all } G \in B(F,r), k \leq m\}.
\end{equation*}
If $M = \infty$, then $B(F,r)$ consists of a single point, so in this case it can be covered by one ball of radius $r/2$.  Otherwise, $r \geq e^{-(M+1)}$ and so $r/2 \geq e^{-(M+2)}$.  Let $\deg(v) \leq \mu$ for $v \in V$.  Hence, if $G \in B(F, r)$, then $G(M) = F(M)$ and there are at most $\mu^2$ possibilities for $G(M+2)$.  If $G, H \in B(F,r)$ are such that $G(M+2) = H(M+2)$, then $d(G,H) \leq e^{-(M+2)} \leq r/2$.  It follows that $B(F,r)$ is contained in $\mu^2$ balls of radius $r/2$.
\end{proof}

The quasisymmetries induced on the end spaces of these trees give rise to quasi-isometries between the trees themselves (Lemma \ref{Tree qi}).  The proof strategy is to show maximal geodesically complete subtrees are quasi-isometric first, and that these subtrees are quasi-isometric to the original trees.  We follow the terminology in \cite{MPR}.

\begin{defn}
A rooted tree $(T,v_0)$ is {\em geodesically complete} if whenever $f \colon [0,t] \to T$ is an isometric embedding with $f(0) = v_0$, there is an isometric embedding $F \colon [0,\infty) \to T$ such that $f(s) = F(s)$ for all $s \in [0,t]$.
\end{defn}

\begin{defn}
Given a rooted tree $(T,v_0)$, we define $(T_\infty,v_0)$ as the unique geodesically complete subtree with the same root $v_0$ that is maximal under inclusion.
\end{defn}

The fact that such a tree exists and is unique follows from an application of Zorn's Lemma, see \cite[Theorem 10.1]{MPM}.  

\begin{remark} \label{Tree qi remark}
We note that if $(T,v_0)$ is visual then from \cite[Proposition 3.8]{MPR} it follows that there is a quasi-isometry $(T,v_0) \to (T_\infty,v_0)$.
\end{remark}

\begin{lemma} \label{Tree qi}
Let $(T,v_0)$ and $(U,w_0)$ be rooted, pseudo-regular, visual trees with bounded degree.  Then $T$ and $U$ are quasi-isometric.
\end{lemma}

It is important that the trees we work with are visual so as to apply Remark \ref{Tree qi remark}.  The construction of a quasi-isometry between $T_\infty$ and $U_\infty$ is given by \cite[Theorem 7.2.1]{BuS}, we provide the main idea of the construction here.

\begin{proof}[Sketch of proof of Lemma \ref{Tree qi}]
Let $\varphi \colon \partial_\infty T \to \partial_\infty U$ be a quasisymmetry.  This exists by Lemma \ref{Tree qs}.  From Remark \ref{Tree qi remark} it suffices to construct a quasi-isometry $f$ between $(T_\infty,v_0)$ and $(U_\infty,w_0)$.  For $v \in T_\infty$ set 
\begin{equation*}
B_v = \{F \in \partial_\infty T_\infty : v \in F([0,\infty))\}
\end{equation*}
and likewise define $B_w$ for $w \in U_\infty$.  Define $f(v) = w$ where $w \in U_\infty$ is a vertex of maximal distance from $w_0$ such that $\varphi(B_v) \subseteq B_{w}$ (such a vertex exists as $B_{w_0} = \partial_\infty U_\infty$).  We then show there is a constant $C>0$ such that if $v,v' \in T_\infty$ with $|v-v'| \leq 1$, then $|f(v) - f(v')| \leq C$.  From our tree structure we may assume without loss of generality that $B_{v'} \subseteq B_v$.  Let $w = f(v)$ and $w' = f(v')$.  We conclude that $B_{w'} \subseteq B_{w}$ and, by using a common point and the quasisymmetry condition, that there is a uniform bound on $|w - w'|$.  Constructing $g \colon (U_\infty,w_0) \to (T_\infty,v_0)$ similarly and checking that $f$ and $g$ are coarse inverses of one another completes the proof.
\end{proof}

We will also work with some homological terminology as in \cite{Wh}.  In what follows we define what is needed, sometimes adapting definitions to our more specific setting.

Given a connected graph $\Gamma=(V,E)$ with bounded degree and equipped with the graph metric, we define a {\em 0-chain} $c$ as a formal sum $c = \sum_{v \in V} c_v v$ with $c_v \in \Z$.  Likewise, a {\em 1-chain} $b$ is a formal sum $b = \sum_{e \in E} b_e e$ with $b_e \in \Z$.  We call a chain {\em bounded} if its coefficients are bounded.  Let $C_0^b (\Gamma)$ denote the set of bounded 0-chains and $C_1^b (\Gamma)$ the set of bounded 1-chains.  Given an orientation on $E$, meaning we view each edge as an ordered pair $e = (e_+, e_-)$, we define the boundary map $\partial \colon C_1^b(\Gamma) \to C_0^b(\Gamma)$ by defining $\partial e = e_+ - e_-$ and extending linearly.

For $r>0$, we define another graph $\Gamma_r = (V_r, E_r)$ where $V_r = V$ and 
\begin{equation*}
E_r = \{(x,y): x,y \in V, \  0 < d(x,y) \leq r\}.
\end{equation*}
This is the 1-dimensional subcomplex of the $r$-Rips complex as defined in \cite{Wh}.  In \cite{Wh} the uniformly finite homology is defined as a limit of the homology formed from $\Gamma_r$ as $r \to \infty$ and the sets of chains are denoted $C_0^{uf}(\Gamma)$ and $C_1^{uf}(\Gamma)$.  This means $C_0^{uf}(\Gamma) = C_0^b(\Gamma)$ and $C_1^{uf}(\Gamma) = \cup_{r > 0} C_1^b(\Gamma_r)$.  Note the uniformly finite homology does not require a graph structure; it can be defined for any $UDBG$ space.

We define an equivalence relation on $C_0^{uf}(\Gamma)$ by setting $c \sim c'$ if and only if there exists $b \in C_1^{uf}(\Gamma)$ such that $\partial b = c - c'$.  We let $[c]$ denote the equivalence class of $c$ under this relation.  The fundamental class $[V]$ is defined as the equivalence class of $\sum_{v \in V} v$.

\begin{remark}
One reason Whyte uses the Rips complex is that he has no graph structure.  In our situation (specifically in a graph with bounded degree), one obtains equivalent homology from the equivalence $c \sim c'$ if and only if there exists $b \in C^b_1(\Gamma)$ such that $\partial b = c - c'$.
\end{remark}

\section{Results from Whyte and Mart\'inez-P\'erez and Rodr\'iguez}
\label{sec results}
Here we state the results from \cite{Wh} and \cite{MPR} relevant for our setting.  We start with Whyte's criteria for promotion of a quasi-isometry to a bilipschitz equivalence.

\begin{thm}[\cite{Wh}, Theorem 4.1]\label{Whyte main}
Let $f \colon X \to Y$ be a quasi-isometry between UDBG spaces with $f_*([X]) =  [Y]$.  Then, there is a bilipschitz map at bounded distance from $f$.
\end{thm}

To apply this, we need a condition that implies $[f_*([X]) - Y]=0$.  This is achieved using an isoperimetric inequality and Theorem \ref{0 chain}.  Given a $UDBG$ metric space $(Z,d)$, a subset $S \subseteq Z$, and $r > 0$, we define the $r$-boundary of $S$ as the set $\partial_r(S) = \{z \in Z : z \notin S \text{ and } d(z,S) \leq r\}$.  Note that for vertex sets of graphs with distances induced from the graph metric, if $A$ is a set of vertices then $\partial_1 A = \partial A$.

\begin{thm}[\cite{Wh}, Theorem 7.6]
\label{0 chain}
Let $Z$ be a $UDBG$ space, $c \in C_0^{uf}(Z)$.  Then, $[c] = 0$ if and only if there are $r, C > 0$ such that for all finite $S \subseteq Z$ we have
\begin{equation*}
\bigl|\sum_S c\bigr| \leq C |\partial_r (S)|.
\end{equation*}
\end{thm}

The main result from \cite{MPR} that concerns us is the following.

\begin{thm}[\cite{MPR}, Theorem 4.15]\label{MPR Theorem}
Let $\Gamma$ be a hyperbolic, rooted, visual graph of bounded degree.  Then, $h(\Gamma) > 0$ if and only if $\partial_\infty \Gamma$ is uniformly perfect for some visual metric.
\end{thm}

Here $\partial_\infty \Gamma$ refers to the Gromov boundary of $\Gamma$ with a visual metric.  While the proof is notationally involved, we summarize the main ideas.  This summary will use results from \cite{MPR}; the exact theorems and lemmas used in their proof can be found in their paper.

\begin{proof}[Summary of proof of Theorem \ref{MPR Theorem}]
It suffices to prove the result for a hyperbolic approximation $\Gamma'$ in place of $\Gamma$.  The boundary at infinity, $\partial_\infty \Gamma = \partial_\infty \Gamma' $ has strongly bounded geometry.  From this and the uniformly perfect condition, one studies the combinatorics of $\Gamma'$.  By using a refinement of $\Gamma'$, one passes to a hyperbolic approximation $\Gamma''$ for which the map $f \colon V_{\Gamma''} \to \R$ defined by $f(v) = k$ for all $v \in V_k$ satisfies conditions which are sufficient to conclude $h(\Gamma) > 0$ (particularly $|\nabla_{xy} f| \leq c_1$ and $\Delta f(x) \geq c_2 > 0$ for some $c_1, c_2 > 0$).
\end{proof}

\section{Proofs of results}
\label{applications}

We now combine the results in Section \ref{sec results} to prove the main theorem. Recall $X$ and $Y$ are assumed to be quasi-isometric, so one of these spaces satisfies a linear isoperimetric inequality if and only if the other one does.

\begin{proof}[Proof of Theorem \ref{main Theorem}]
Let $f \colon X \to Y$ be a quasi-isometry.  We may assume $f \colon V_X \to V_Y$.  By assumption $Y$ supports a linear isoperimetric inequality with constant $h = h(Y)^{-1} > 0$.  That is, for any finite set $S \subseteq V_Y$, we have $|S| \leq h|\partial S|$.  As $f$ is a quasi-isometry and $X$ has bounded degree, there is an $A>0$ such that for each vertex $v \in V_Y$, we have $|(f_*([X]) - [Y]) (v)| \leq A$.  Hence, for $S \subseteq V_Y$ finite,
\begin{equation*}
|\sum_S (f_*([X]) - [Y])| \leq A |S| \leq Ah |\partial S| = Ah |\partial_1 S|
\end{equation*}
and so  $[f_*([X]) - [Y]] = 0$ by Theorem \ref{0 chain} with $C = Ah$ and $r = 1$.  Thus, by Theorem \ref{Whyte main}, $f$ is within bounded distance of a bilipschitz equivalence.
\end{proof}

We now prove the corollaries of Theorem \ref{main Theorem} and discuss the conditions in Corollary \ref{Cor 1}.

\begin{proof}[Proof of Corollary \ref{Cor 1}]
By Lemma \ref{Tree qi} there exists a quasi-isometry $f \colon X \to Y$ which maps $V_X$ to $V_Y$.  As trees are hyperbolic, we may apply Theorem \ref{MPR Theorem} to conclude that $Y$ supports a linear isoperimetric inequality.   Thus, the result follows from Theroem \ref{main Theorem}.
\end{proof}

We examine the conditions in Corollary \ref{Cor 1}, namely that $X$ and $Y$ are pseudo-regular, visual trees of bounded degree.  Recall the fact that our trees were visual was important for Lemma \ref{Tree qi}.  There is a similar condition in \cite[Theorem 3.16]{MPR} which must be satisfied for our tree to even support a linear isoperimetric inequality.  Pseudo-regularity guarantees that our trees branch regularly; without this condition, one tree, say $X$, will have arbitrarily long segments with no branching.  If this is not the case for $Y$, then the size of the number of vertices in balls of radius $R$ in $Y$ grow exponentially in $R$.  If a bilipschitz map $g$ existed between vertices with bilipschitz constant $C>0$, then one could consider a nonbranching segment of length $M$ in $X$.  Letting $x$ denote its midpoint vertex, the ball $B(g(x),R)$ has at least $c^R$ vertices for some $c>1$ that depends only on $Y$.    Assuming $M > 3CR$, there are at most $2CR$ vertices on our nonbranching segment that could be the image of these $c^R$ vertices.  As $2CR/c^R < 1$ for large $R$, it follows that $g$ cannot exist.  Indeed, similar reasoning shows that such trees cannot even be quasi-isometric.

The bounded degree condition is similar; if the degree of $Y$ is at most $\mu < \infty$, then the size of any ball $B(y, R)$ is bounded by $c\mu^{R+1}$ for some $c > 0$.  Hence, if $X$ had unbounded degree and a bilipschitz equivalence $g$ existed, we could arrive at a contradiction by considering $g(B(x_n,1))$ for a sequence of vertices $x_n$ with strictly increasing degree in $X$.

\begin{proof}[Proof of Corollary \ref{Cor 2}]
Note that Ahlfors regular metric spaces are uniformly perfect.  By a similar construction to that in Lemma \ref{Tree qi}, it is known that $X$ and $Y$ are quasi-isometric (see \cite[Theorem 7.2.1]{BuS} for a detailed proof or \cite[Lemma 3.6]{Li} for a summarized proof).  Theorem \ref{MPR Theorem} shows $Y$ satisfies a linear isoperimetric inequality, so Theorem \ref{main Theorem} applies.
\end{proof}

\end{document}